\documentclass[a4paper,11pt,reqno]{amsart}
\usepackage{amsmath}
\usepackage{amssymb}
\usepackage{amsthm}
\usepackage[english]{babel}
\usepackage{hyperref}
\usepackage{mathrsfs}
\usepackage{tikz}
\usepackage{enumitem}
\usepackage[capitalise]{cleveref}

\numberwithin{equation}{section}

\newtheorem{thm}[equation]{Theorem}
\crefname{thm}{Theorem}{Theorems}
\Crefname{thm}{Theorem}{Theorems}
\newtheorem{lemma}[equation]{Lemma}
\crefname{lemma}{Lemma}{Lemmas}
\Crefname{lemma}{Lemma}{Lemmas}

\crefname{cor}{Corollary}{Corollaries}
\Crefname{cor}{Corollary}{Corollaries}

\crefname{defi}{Definition}{Definitions}
\Crefname{defi}{Definition}{Definitions}

\crefname{rem}{Remark}{Remarks}
\Crefname{rem}{Remark}{Remarks}

\newcommand{\s}{{\sf S}}

\newcommand{\Md}{\!\mod}

\newcommand{\F}{\mathbb{F}}

\newcommand{\sgn}{\mathbf{\mathrm{sgn}}}

\renewcommand{\epsilon}{\varepsilon}
\renewcommand{\phi}{\varphi}
\newcommand{\xymat}{\xymatrix@R=6pt@C=10pt}

\newcommand{\la}{\lambda}
\newcommand{\be}{\beta}
\newcommand{\al}{\alpha}
\newcommand{\eps}{\epsilon}
\newcommand{\ga}{\gamma}
\newcommand{\de}{\delta}
\newcommand{\si}{\sigma}

\newcommand{\da}{{\downarrow}}
\newcommand{\ua}{{\uparrow}}

\def\Par{{\mathscr {P}}}

\def\C{{\mathbb {C}}}

\newcommand{\dbl}{\mathbf{\mathrm{dbl}}}
\newcommand{\bdbl}{\overline{\mathbf{\mathrm{dbl}}}}

\def\Md#1{\text{ }(\text{\rm mod } #1)\,}

\def\c{\overline{c}}

\begin{document}

\title[Row and column removal results for spin representations]{Generalised row and column removal results for decomposition numbers of spin representations of symmetric groups in characteristic 2}

\author{Lucia Morotti}
\address{Department of Mathematics, University of York, York, YO10 5DD, UK}
\email{lucia.morotti@york.ac.uk}

\thanks{The author thanks the referee for helpful comments.\\
The author was supported by the Royal Society grant URF$\backslash$R$\backslash$221047. While starting to work on this paper the author was working at the Mathematisches Institut of the Heinrich-Heine-Universität D\"usseldorf.}

\begin{abstract}
We prove generalised row and column removal results for decomposition numbers of spin representations of symmetric groups in characteristic 2, similar to Donkin's results for decomposition numbers of representations of symmetric groups.
\end{abstract}

\maketitle

\section{Introduction}

Row and column removal results for decomposition numbers of symmetric groups $\s_n$ were obtained by James \cite{j3} and later generalised by Donkin \cite{d}, allowing, in certain cases, to obtain decomposition numbers from decomposition numbers of smaller symmetric groups.

Following \cite{JamesBook}, for a partition $\la$ of $n$ let $S^\la$ be the corresponding Specht module, that is the irreducible representation of $\C\s_n$ labeled by $\la$, while for $\mu$ a $p$-regular partition on $n$ let $D^\mu$ be the irreducible representation of $\overline{\F}_p\s_n$ labeled by $\mu$. James' row and column removal results \cite[Theorems 5, 6]{j3} state that the multiplicity of $D^\mu$ in the reduction modulo $p$ of $S^\la$ is equal to the multiplicity of $D^{\overline{\mu}}$ in the reduction modulo $p$ of $S^{\overline{\la}}$, where $\overline{\la}$ and $\overline{\mu}$ are obtained from $\la$ and $\mu$ by removing their first rows or columns, provided the first rows or columns of $\la$ and $\mu$ have the same length.

In order to state the generalised versions of these results, we need to introduce some notation. For any $n\geq 0$ and $p$ a prime, let $\Par(n)$ be the set of partitions of $n$ and $\Par_p(n)$ the set of $p$-regular partitions of $n$. For a given partition $\la$, let $\la'$ denote its conjugate partition. Further, for any $m\geq 0$, let $\la^{r\leq m}$ be the partition consisting of the first $m$ rows of $\la$ and $\la^{r>m}$ be the partitions consisting of rows $m+1,m+2,\ldots$ of $\la$. Similarly define $\la^{c\leq m}$ and $\la^{c>m}$ considering columns instead of rows. Such restrictions on which rows and columns are considered can be combined, for example $\la^{r>m,c\leq \ell}$ consists of nodes of $\la$ which are in rows larger then $m$ and columns up to $\ell$ (up to identifying partitions and their Young diagrams and shifting the new diagram by $m$ rows to indeed obtain the Young diagram of a partition). 

With these notations set, we are now able to state Donkin's generalised row and column removal results for decomposition numbers:

\begin{thm}\cite[Theorem 5]{d}\label{T1}
Let $n\geq 0$, $m\geq 1$, $\la\in\Par(n)$ and $\mu\in\Par_p(n)$. If $|\la^{r\leq m}|=|\mu^{r\leq m}|$ then
\[[S^\la:D^{\mu}]=[S^{\la^{r\leq m}}:D^{\mu^{r\leq m}}]\cdot[S^{\la^{r>m}}:D^{\mu^{r>m}}].\]
\end{thm}

\begin{thm}\cite[Theorem 6]{d}\label{T2}
Let $n\geq 0$, $m\geq 1$, $\la\in\Par(n)$ and $\mu\in\Par_p(n)$. If $|\la^{c\leq m}|=|\mu^{c\leq m}|$, $\mu\unrhd\la$ and $b=\la'_m$, then
\[[S^\la:D^{\mu}]=[S^{\la^{r>b,c\leq m}}:D^{\mu^{r>b,c\leq m}}]\cdot[S^{\la^{c>m}}:D^{\mu^{c>m}}].\]
\end{thm}

The cases with $m=1$ are James' results (for the column removal version the assumption $\mu\unrhd\la$ could be dropped in this case due to well known results on the shape of the decomposition matrices of symmetric groups).

Let $\widetilde\s_n$ be a double cover of the symmetric group $S_n$. Faithful irreducible representations of $\widetilde\s_n$ are called spin representations. In characteristic $0$, (pairs of) irreducible spin representations of $\s_n$ are labeled by strict partitions of $n$, as first shown in \cite{schur}. For $\la$ such a partition we write $S(\la,0)$ or $S(\la,\pm)$ for the corresponding representations. When working in characteristic 2, we will identify $S(\la,0)$ and $S(\la,\pm)$ with their reduction modulo $2$. Reductions modulo $2$ of spin representations can be viewed as characteristic $2$ representations of the corresponding symmetric group, thus their composition factors are the modules $D^\mu$ for 2-regular partitions $\mu$. In odd characteristic reductions modulo $p$ of spin representations remain spin representations. This explains why the two ($p=2$ and $p$ odd) cases are studied separately.

Results on the shape of the decomposition matrices of spin representations of symmetric groups in characteristic 2 were obtained by Benson \cite{ben} and later improved by Bessenrodt and Olsson \cite{bo}. In particular Benson showed that, provided $\dbl(\la)$ is $2$-regular, the multiplicity of $D^{\dbl(\la)}$ in $S(\la,\eps)$ is a certain specific power of $2$. Here, for $\la=(\la_1,\ldots,\la_h)$ with $\la_h>0$, its double $\dbl(\la)$ is the partition
\[\dbl(\la)=(\lceil (\la_1+1)/2\rceil,\lfloor (\la_1-1)/2\rfloor),\ldots,\lceil (\la_h+1)/2\rceil,\lfloor (\la_h-1)/2\rfloor)).\]
This result, together with results on spin decomposition numbers in characteristic $2$ for spin representations labeled by partitions with at most $2$-parts (see \cite{m3}), shows evidence of row and column removal results for decomposition numbers of spin representations in characteristic 2 and thus leads to study them. These removal results, which are stated below, are similar to the above stated results for (generalised) row and columns removal of decomposition numbers of symmetric groups. Doubling and powers of 2 appearing in \cite{ben,bo} are reflected in the row and column removal results for spin representations.

In the next two theorems, $\delta,\eps,\eps'\in\{0,\pm\}$ are determined by the partitions labeling the corresponding spin representations. As $S(\mu,+)$ and $S(\mu,-)$ reduce in the same way in characteristic 2 whenever $\mu$ labels two irreducible spin representations in characteristic 0 by \cite[Theorem 7.1]{s}, the choices of representative, if needed, does not affect the decomposition numbers in the next two theorems.

\begin{thm}\label{t1}
Let $p=2$, $n\geq 0$, $m\geq 1$ and $\la,\mu\in\Par_2(n)$. If $|\la^{r\leq m}|=|\mu^{r\leq 2m}|$ then
\[[S(\la,\de):D^{\mu}]=2^a[S(\la^{r\leq m},\eps):D^{\mu^{r\leq 2m}}]\cdot[S(\la^{r>m},\eps'):D^{\mu^{r>2m}}]\]
with $a=1$ if $\eps,\eps'\in\{\pm\}$ or $a=0$ else.
\end{thm}

\begin{thm}\label{t2}
Let $p=2$, $n\geq 0$, $m\geq 1$ and $\la,\mu\in\Par_2(n)$. If $|\la^{c\leq 2m}|=|\mu^{c\leq m}|$, $\mu\unrhd\bdbl(\la)$ and $b=\la'_{2m}$, then
\[[S(\la,\de):D^{\mu}]=2^a[S(\la^{r>b,c\leq 2m},\eps'):D^{\mu^{r>2b,c\leq m}}]\cdot[S(\la^{c>2m},\eps):D^{\mu^{c>m}}]\]
with $a=1$ if $\eps,\eps'\in\{\pm\}$ or $a=0$ else.
\end{thm}

In view of known results on the shapes of decomposition matrices, 
it can be checked that the assumptions $\mu\unrhd\la$ in Theorem \ref{T2} resp. $\mu\unrhd\bdbl(\la)$ in Theorem \ref{t2} can be dropped provided the multiplicities $[S^{\la^{r>b,c\leq m}}:D^{\mu^{r>b,c\leq m}}]$ resp. $[S(\la^{r>b,c\leq 2m},\eps'):D^{\mu^{r>2b,c\leq m}}]$ are set to be 0 if the labeling partitions have different sizes or if the labeling partitions ${\mu^{r>b,c\leq m}}$  resp. ${\mu^{r>2b,c\leq m}}$ are not $p$-regular resp. not $2$-regular.

We also note that versions of \cref{t1,t2} which are not taking in considerations doubling of partitions (that is removing $m$ rows or $m$ columns from both $\la$ and $\mu$) are wrong, even after possible rescaling by a non-zero scalar. For example it can be checked combining \cite{ben,Wales} that $[S((n),\eps):D^{(n)}]$ is 0 if $n\geq 3$ while it is 1 for $n\leq 2$.

In \cref{s2,s3} we will introduce some notations and some basic results on decomposition numbers. In \cref{s4} we will write reductions modulo 2 of basic spin representations as linear combinations of certain modules of symmetric groups. In \cref{lr,st} we will study Littlewood-Richardson coefficients and shifted tableaux and use these results to get information on certain tensor products with permutation modules and basic spin representations. We will then prove \cref{t1,t2} in \cref{st1,st2}. Though the results obtained are similar, the methods used in this paper are rather different than those used in \cite{d,j3}.


\section{Basic notations}\label{s2}

As in the introduction, for $n\geq 0$, let $\s_n$ be the symmetric group on $n$ letters and $\widetilde{\s}_n$ be a double cover of $\s_n$. Double covers of symmetric groups and their representations have been studied by Schur \cite{schur}, see also \cite{s}. By definition there exists $z\in\widetilde{\s}_n$ central of order 2 with $\widetilde{\s}_n/\langle z\rangle\cong\s_n$. We will be working with representations over the field $\C$ when working in characteristic $0$ and over $\overline\F_2$ when working in characteristic $2$.

Since $z$ is central of order 2, it acts as $\pm 1$ on any irreducible representation. Irreducible representations of $\widetilde{\s}_n$ on which $z$ acts as $1$ can be naturally viewed as representations of $\s_n\cong\widetilde{\s}_n/\langle z\rangle$. Representations of $\widetilde{\s}_n$ on which $z$ acts as $-1\not=1$ are called spin representations. Note that $\s_n$ has no irreducible spin representation in characteristic 2  and that irreducible 
spin representations in characteristic 0 can always be viewed as representations of $\s_n$ when reduced modulo 2.

It is well know that irreducible representations of $\s_n$ are labeled by partitions in characteristic $0$ and by $p$-regular partitions (that is partitions with no part repeated $p$ or more times) in characteristic $p$, see for example \cite{JamesBook}. In characteristic 0, irreducible spin representations, or pairs of them, are labeled by strict partitions, that is partitions with no repeated part, see \cite{schur,s}. We define $\Par(n)$ to be the set of partitions of $n$ and $\Par_2(n)$ to be the set of $2$-regular (or strict) partitions of $n$.

As in the introduction and \cite[\S4, \S7]{JamesBook}, for $\la\in\Par(n)$ we let the Specht module $S^\la$ to be the irreducible representation of $\C\s_n$ labeled by $\la$. We will identify $S^\la$ also with its reduction modulo 2. Similarly for $\mu\in\Par_2(n)$ we let $D^\mu$ be the irreducible representation of $\overline{\F}_2\s_n$ indexed by $\mu$. For $\nu$ any composition of $n$ we also define $M^\nu$ to be the permutation module $1\ua_{\s_\nu}^{\s_n}$, with
\[\s_{\nu}\cong\s_{\{1,\ldots,\nu_1\}}\times\s_{\{\nu_1+1,\ldots,\nu_1+\nu_2\}}\times\ldots\subseteq\s_n.\]

As mentioned above, strict partitions label either one or two irreducible spin representations in characteristic 0. In order to distinguish between the two cases, we say that a partition is even if it has an even number of even parts and that it is odd otherwise. Thus the parity of $\la$ is just the parity of $|\la|-h(\la)$, with $|\la|=\la_1+\la_2+\ldots$ the size of $\la$ and $h(\la)$ the number of parts of $\la$. We also define $a(\la):=0$ if $\la$ is even or $a(\la):=1$ if $\la$ is odd. To keep notation simpler, we will drop double parenthesis. For example, we will write $a(n)$ for $a((n))$.

By \cite[Theorem 7.1]{s} or \cite[p. 235]{schur}, if $\la\in\Par_2(n)$ is even then $\la$ labels a unique spin representation $S(\la,0)$ of $\C\s_n$. If instead $\la\in\Par_2(n)$ is odd, then $\la$ labels 2 irreducible spin representations $S(\la,\pm)$ of $\C\s_n$. In this case we also have that $S(\la,+)\cong S(\la,-)\otimes\sgn$ with $\sgn$ the sign representation of $\s_n$, see \cite[Theorem 7.1]{s}. In order to keep the formulation of some of the following results easier, for $\la\in\Par_2(n)$ we define modules $S(\la)$ to be either $S(\la,0)$ is $\la$ is even or $S(\la,+)\oplus S(\la,-)$ is $\la$ is odd. 
When working in characteristic 2, we identify the modules $S(\la,0)$ or $S(\la,\pm)$ as well as the modules $S(\la)$ with their reductions modulo $2$.

We also introduce some further notations for partitions which will be used in the paper. We identify partitions and their Young diagrams, that is the set of nodes $\{(i,j):1\leq j\leq\al_i\}$ for $\al$ a partition. Further we define $\al'$ to be the conjugate partition of $\al$, that is the partition obtained by transposing the Young diagram of $\al$ across the main diagonal. We also use the English notation for Young diagrams, that is rows increases when moving down (this will play a role when studying tableaux in \cref{lr,st} and speaking about nodes above or below other nodes). Further for $\al,\be$ partitions, we write $\al\sqcup\be$ for the partition obtained by rearranging the parts of the composition $(\al_1,\ldots,\al_{h(\al)},\be_1,\ldots,\be_{h(\be)})$ in weakly decreasing order and we write $\al+\be$ for the partition $(\al_1+\be_1,\al_2+\be_2,\ldots)$. We also let $\unrhd$ is the dominance order on partitions.

In the introduction we gave the definition of double of a partition that was used in \cite{ben}. In \cite{bo} Bessenrodt and Olsson defined a different double $\bdbl(\la)$ of a partition $\la$ through
\begin{align*}
\bdbl(\la)&=(\lceil \la_1/2\rceil,\lfloor \la_1/2\rfloor),\ldots,\lceil \la_{h(\la)}/2\rceil,\lfloor \la_{h(la)}/2\rfloor)).
\end{align*}
Note that in general $\dbl(\la)\not=\bdbl(\la)$ and that $\dbl(\la)$ (as defined in the introduction) and $\bdbl(\la)$ are not necessarily partitions but only compositions. However $\bdbl(\la)$ is a partition if $\la$ is strict (or more generally when $\la$ has no repeated odd part).

As we will use both inner and outer tensor products of representations, we will use notation $V\otimes W$ for inner tensor products and $V\boxtimes W$ for outer tensor products.

\section{Decomposition numbers}\label{s3}

In this section we review some known results on decomposition numbers which will be used in the proofs of \cref{t1,t2}. We start with results on the shape of the decomposition matrices for representations and spin representations of symmetric groups in characteristic 2. 

\begin{lemma}\cite[Corollary 12.2]{JamesBook}\label{reg} 
Let $p=2$, $\la\in\Par(n)$ and $\mu\in\Par_2(n)$. If $[S^\la:D^\mu]>0$ then $\mu\unrhd\la$. Further if $\la\in\Par_2(n)$ then $[S^\la:D^\la]=1$.
\end{lemma}

\begin{lemma}\label{regs}\cite[Theorem 1.2]{ben}
Let $p=2$, $\la,\mu\in\Par_2(n)$. If $[S(\la,\eps):D^\mu]>0$ then $\mu\unrhd\dbl(\la)$. Further if $\dbl(\la)\in\Par_2(n)$ then
\[[S(\la,\eps):D^{\dbl(\la)}]=2^{(h_2(\la)-a(\la))/2},\]
where $h_2(\la)$ is the number of even parts of $\la$.
\end{lemma}

We need one more result, which compares composition factors of $S(\la)$ and $S(\la,\eps)$, where $\eps=0$ or $\pm$ depending on the partity of $\la$.

\begin{lemma}\label{L151123_2}
Let $p=2$ and $\la,\mu\in\Par_2(n)$. Then
\[[S(\la,\eps):D^\mu]=2^{-a(\la)}[S(\la):D^\mu].\]
\end{lemma}

\begin{proof}
If $a(\la)=0$ then $S(\la)=S(\la,0)$. If instead $a(\la)=1$ then $S(\la)=S(\la,+)\oplus S(\la,-)$. Further in this case $S(\la,+)\cong S(\la,-)\otimes\sgn$, in particular $S(\la,+)$ and $S(\la,-)$ reduce in the same way in characteristic 2.
\end{proof}

\section{Linear combinations for basic spin modules}\label{s4}

Proofs of \cref{t1,t2} rely on certain linear combinations of reductions modulo 2 of the modules $S((n))$ (basic spin modules) in terms of the permutation modules $M^{(n-h,h)}$ (such linear combinations exist only in characteristic $2$). We collect these results here. In this section, as well in other places in the paper, when writing modules as linear combinations of other modules, we will work in the Grothendieck group.

We start with a result comparing $S((n))$ and $D^{\dbl(n)}$. 

\begin{lemma}\label{bs}
Let $p=2$ and $n\geq 1$. Then
\[[S((n))]=2^{a(n)}[D^{\dbl(n)}].\]
\end{lemma}

\begin{proof}
For $n\geq 4$ this holds by \cite[Table III]{Wales} and \cref{regs}. For $n\leq 3$ use \cref{regs} together with the fact that there are no further module in the same block in characteristic 2.
\end{proof}

In view of Lemma \ref{bs}, we call also the modules $D^{\dbl(n)}$ basic spin modules. We next write these modules in terms of the Specht modules $S^{(n-h,h)}$ indexed by partitions with at most 2 parts.

\begin{lemma}\label{bss}
Let $p=2$ and $n\geq 1$ and set $S^{(n-h,h)}:=0$ if $h<0$. Then
\begin{enumerate}
\item
If $n$ is even then
\[[D^{(n/2+1,n/2-1)}]=\sum_{i\geq 0}([S^{(n/2+4i+1,n/2-4i-1)}]-[S^{(n/2+4i+2,n/2-4i-2)}]).\]

\item
If $n$ is odd then
\[[D^{((n+1)/2,(n-1)/2)}]=\sum_{i\geq 0}([S^{((n+1)/2+4i,(n-1)/2-4i)}]-[S^{((n+5)/2+4i,(n-5)/2-4i)}]).\]
\end{enumerate}
\end{lemma}

\begin{proof}
Set also $D^{(n-k,k)}:=0$ if $k<0$. Further for $a,b\geq 0$ with $a=\sum_{i\geq 0}a_i2^i$ and $b=\sum_{i\geq 0}b_i2^i$ their 2-adic decompositions, let $g_{a,b}=1$ if $b_i\in\{0,a_i\}$ for each $i$ and let $g_{a,b}=0$ otherwise. If $a<0$ or $b<0$ define $g_{a,b}:=0$.

By \cite[Theorem 24.15]{JamesBook}, for each $0\leq h\leq\lfloor n/2\rfloor$ and $0\leq k\leq\lfloor (n-1)/2\rfloor$,
\[[S^{(n-h,h)}]=\sum_{k}g_{n-2k+1,h-k}[D^{(n-k,k)}]\]
(note that since $n-2k+1>2(h-k)\geq 0$ there always exists $i$ with $n-2k+1\geq 2^i>h-k$).

Assume first that $n=2m$ is even. Then for every $i\geq 0$
\begin{align*}
&\sum_{i\geq 0}[S^{(m+4i+1,m-4i-1)}]-[S^{(m+4i+2,m-4i-2)}]\\
&=\sum_{k\leq m-1}\sum_{i\geq 0}(g_{2m-2k+1,m-4i-k-1}-g_{2m-2k+1,m-4i-k-2})[D^{(2m-k,k)}]
\end{align*}
(as $g_{n-2k+1,h-k}=0$ if $k>h$ and by definition $S^{(n-h,h)}=0$ or $D^{(n-k,k)}=0$ if $h<0$ or $k<0$).

If $m-k=2a$ then
\begin{align*}
&g_{2m-2k+1,m-4i-k-1}-g_{2m-2k+1,m-4i-k-2}\\
&=g_{4a+1,2a-4i-1}-g_{4a+1,2a-4i-2}\\
&=g_{2\cdot2a+1,2(a-2i-1)+1}-g_{2\cdot2a+1,2(a-2i-1)}\\
&=g_{2a,a-2i-1}-g_{2a,a-2i-1}=0.
\end{align*}
If $m-k=4a+3$ then
\begin{align*}
&g_{2m-2k+1,m-4i-k-1}-g_{2m-2k+1,m-4i-k-2}\\
&=g_{8a+7,4a-4i+2}-g_{8a+7,4a-4i+1}\\
&=g_{4(2a+1)+2+1,4(a-i)+2}-g_{4(2a+1)+2+1,4(a-i)+1}\\
&=g_{2a+1,a-i}-g_{2a+1,a-i}=0.
\end{align*}
If $m-k=4a+1$ then
\begin{align*}
&g_{2m-2k+1,m-4i-k-1}-g_{2m-2k+1,m-4i-k-2}\\
&=g_{8a+3,4a-4i}-g_{8a+3,4a-4i-1}\\
&=g_{4\cdot2a+2+1,4(a-i)}-g_{4\cdot2a+2+1,4(a-i-1)+2+1}\\
&=g_{2a,a-i}-g_{2a,a-i-1}.
\end{align*}
So
\begin{align*}
&\sum_{i\geq 0}[S^{(m+4i+1,m-4i-1)}]-[S^{(m+4i+2,m-4i-2)}]\\
&=\sum_a\sum_{i\geq 0}(g_{2a,a-i}-g_{2a,a-i-1})[D^{(m+4a+1,m-4a-1)}]\\
&=\sum_ag_{2a,a}[D^{(m+4a+1,m-4a-1)}].
\end{align*}
If $a\not=0$ then $g_{2a,a}=0$. On the other hand $g_{0,0}=1$, so that (i) holds.

Assume now that $n=2m-1$ is odd. In view of \cite[Theorem 11.2.7]{KBook} we have that $D^{(m+1,m-1)}\da_{\s_{2m-1}}\cong D^{(m,m-1)}$ and by \cite[Theorem 9.2]{JamesBook} that $[S^{(2m-h,h)}\da_{\s_{2m-1}}]=[S^{(2m-h,h-1)}]+[S^{(2m-h-1,h)}]$ for each $h<m$. Thus (ii) follows from (i).
\end{proof}

The next two results now allow to write basic spin modules in terms of the permutation modules $M^{(n-h,h)}$.

\begin{thm}\label{bsm}
Let $p=2$ and $n\geq 1$. Then
\begin{enumerate}
\item
If $n$ is even then $[D^{(n/2+1,n/2-1)}]=\sum_{h=0}^{n/2-1}r_{h}[M^{(n/2+1+h,n/2-1-h)}]$ with $r_{h}=1$ if $h$ is even, $r_{h}=-2$ if $h\equiv 1\Md{4}$ and $r_{h}=0$ if $h\equiv 3\Md{4}$.

\item
If $n$ is odd then $[D^{((n+1)/2,(n-1)/2)}]=\sum_{h=0}^{(n-1)/2}s_{h}[M^{((n+1)/2+h,(n-1)/2-h)}]$ with $s_{h}=1$ if $h\equiv 0\text{ or }3\Md{4}$ and $s_{h}=-1$ if $h\equiv 1\text{ or }2\Md{4}$.
\end{enumerate}
\end{thm}

\begin{proof}
This holds by \cref{bss} and the fact that $[M^{(n-h,h)}]=\sum_{k\leq h}[S^{(n-k,k)}]$ for $0\leq h\leq\lfloor n/2\rfloor$ (see for example \cite[Example 17.17]{JamesBook}).
%
\end{proof}

\begin{thm}\label{L111223_3}
Let $p=2$, $1\leq\ell<n$ and $t_{n,h}$ with $[D^{\dbl(n)}]=\sum_{h=0}^nt_{n,h}[M^{(n-h,h)}]$.
Then
\begin{align*}
&2^{1-(1-a(\ell))(1-a(n-\ell))}[D^{\dbl(\ell)}\boxtimes D^{\dbl(n-\ell)}].\\
&=\sum_{h=0}^{n}\sum_{j=\max\{0,h+\ell-n\}}^{\min\{h,\ell\}} t_{n,h}
[M^{(\ell-j,j)}\otimes M^{(n-\ell-h+j,h-j)}].
\end{align*}
\end{thm}

\begin{proof}
By \cref{bs}, $[S((n))]=2^{a(n)}[D^{\dbl(n)}]$ and by \cite[Theorem 8.1]{s}, $[S((n))\da_{\widetilde{\s}_{m}}]=c[S((m))]$ for some $c>0$. So any composition factor of $D^{\dbl(n)}\da_{\s_{\ell}}$ resp. $D^{\dbl(n)}\da_{\s_{n-\ell}}$ is of the form $D^{\dbl(\ell)}$ resp. $D^{\dbl(n-\ell)}$ and then
\begin{align*}
&[D^{\dbl(n)}\da_{\s_{\ell,n-\ell}}]\\
&=d[D^{\dbl(\ell)}\boxtimes D^{\dbl(n-\ell)}]
\end{align*}
for some $d$. By \cref{bs} and \cite[Theorem 3.3]{s}, $D^{\dbl(m)}$ has dimension $2^{\lfloor(m-1)/2\rfloor}$ for every $m\geq 1$. Comparing dimensions it follows that $d=2^{1-(1-a(\ell))(1-a(n-\ell))}$.

As
\[[M^{(n-a,a)}\da_{\s_{\ell,n-\ell}}]=\sum_{j=\max\{0,a+\ell-n\}}^{\min\{a,\ell\}}[M^{(\ell-j,j)}\boxtimes M^{(n-\ell+j-a,a-j)}]\]
by Mackey's induction-restriction theorem, the theorem follows.
\end{proof}

\section{Littlewood-Richardson coefficients}\label{lr}

A tableau $t$ is a map from a set of nodes to the ordered set $\{1<2<\ldots\}$. The corresponding set of nodes is called the shape of the tableau $t$. For $t$ any tableau we also define its content to be the composition whose $i$-th part is the number of nodes of content $i$ in $t$. The shape of a tableau is the underlying set of nodes.

This notion of content is consistent with notation from \cite{s} and other literature when working with tableaux. It is however different from notions of content used to consider blocks of representations of (double covers of) symmetric groups in positive characteristic.

For $t$ a tableau let $w(t)$ be the word obtained by concatenating the contents of nodes of $t$, reading row by row, starting from the top and within a row from the right. For any word $w=w_1,\ldots,w_\ell$, any $i\geq 1$ and $0\leq j\leq \ell$ define $m_i(j)$ to be the number of $i$ in $w_1\ldots w_j$.

We say that a tableau is semistandard if entries increase along columns and weakly increase along rows. We say that a tableau $t$ satisfies the lattice condition if $m_i(j)\geq m_{i+1}(j)$ for every $i\geq 1$ and $0\leq j\leq\ell$ (with $\ell$ the number of nodes in the tableau).

If $t$ is semistandard it follows from the definition that $t$ satisfies the lattice condition if and only if for every $i,k\geq 1$, if $t$ has at least $k$ nodes of content $i+1$ then $t$ has at least $k$ nodes of content $i$ and the $k$-th $i$ appears before the $k$-th $i+1$, counting nodes according to their position in $w(t)$.

Let $\al,\be,\ga$ be partitions. If $\al\subseteq\ga$ define the Littlewood-Richardson coefficient $c_{\al,\be}^\ga$ to be the number of semistandard tableau of shape $\ga\setminus\al$ and content $\be$ which satisfy the lattice condition. If $\al\not\subseteq\ga$ define instead $c_{\al,\be}^\ga:=0$.

The next result allows to describe tensor products of Specht modules and permutation modules in terms of Littlewood-Richardson coefficients:

\begin{lemma}\label{L101123_3}
Let $\al,\be\in\Par(n)$ and $0\leq i\leq\lfloor n/2\rfloor$. Then
\[[S^\al\otimes M^{(n-i,i)}:S^\be]=\sum_{\ga\in\Par(n-i),\de\in\Par(i)}c_{\ga,\de}^\al c_{\ga,\de}^\be.\]
\end{lemma}

\begin{proof}
This follows by  \cite[16.1 and 16.4]{JamesBook} and by Frobenius reciprocity in characteristic 0.
\end{proof}

\begin{lemma}\label{L111223_2}\cite[16.5]{JamesBook}
For any partitions $\al,\be,\ga$ with $|\al|+|\be|=|\ga|$ we have that $c_{\al,\be}^\ga=c_{\be,\al}^\ga=c_{\al',\be'}^{\ga'}$.
\end{lemma}

The next lemmas give conditions for (products of) Littlewood-Richardson coefficients to be non-zero. At least \cref{L101123_4} is already known. As we are unaware where a proof may be found prove the result again here.

\begin{lemma}\label{L101123_4}
Let $\al,\be,\ga$ be partitions with $c_{\al,\be}^\ga\geq 1$. Then $\al,\be$ and $\al\sqcup\be\unlhd\ga\unlhd\al+\be$.
\end{lemma}

\begin{proof}
As $c_{\al,\be}^\ga\geq 1$ we have $\al\subseteq\ga$. Then in any semistandard tableau of shape $\ga\setminus\al$ which satisfy the lattice property, nodes in rows $i$ can only have content $j$ with $j\leq i$. So $\ga\unlhd\al+\be$. The lemma then follows by \cref{L111223_2}, since $\al\sqcup\be=(\al'+\be')'$ and the dominance order is reversed when conjugating partitions (see for example \cite[p. 9]{JamesBook}).
\end{proof}

\begin{lemma}\label{L131123}
Let $\ga,\de\in\Par(n)$ be such that there exist $0\leq i\leq n$, $\al\in\Par(n-i)$ and $\be\in\Par(i)$ with $c^\ga_{\al,\be}c^\de_{\al,\be}>0$. Then $\de\unrhd\bdbl(\la)$.
\end{lemma}

\begin{proof}
By Lemma \ref{L101123_4}, $\al+\be\unrhd\ga$. Further $\de\unrhd\al\sqcup\be$, so that $\de\unrhd(\al_1,\be_1,\al_2,\be_2,\ldots)$. In particular, for each $k\geq 1$,
\[\sum_{j=1}^{2k}\de_j\geq\sum_{j=1}^k(\al_j+\be_j)\geq\sum_{j=1}^k\ga_j.\]
So $\overline{\de}=(\de_1+\de_2,\de_3+\de_4,\ldots)\unrhd\ga$. Since $\de$ is a partition it follows that $\de\unrhd\bdbl(\overline{\de})\unrhd\bdbl(\ga)$.
\end{proof}

The next lemmas compare Littlewood-Richardson coefficients for certain triples of partitions. 

\begin{lemma}\label{L111223}
Let $\al,\be,\ga$ be partitions with $|\al|+|\be|=|\ga|$. If $m\geq 0$ and $|\al^{r\leq m}|+|\be^{r\leq m}|=|\ga^{r\leq m}|$, then $c_{\al,\be}^{\ga}=c_{\al^{r\leq m},\be^{r\leq m}}^{\ga^{r\leq m}}c_{\al^{r>m},\be^{r>m}}^{\ga^{r>m}}$.
\end{lemma}

\begin{proof}
By definition $\al\subseteq\ga$ if and only if $\al_i\leq\ga_i$ for all $i$ if and only if $\al^{r\leq m}\subseteq\ga^{r\leq m}$ and $\al^{r>m}\subseteq\ga^{r>m}$.

Assume that the above conditions holds (else both sides are 0). Let $t$ be any semistandard tableau of shape $\ga\setminus \al$ and content $\be$ which satisfies the lattice condition. Then any node in the first $m$ rows of $t$ has content $\leq m$. Since $|\al^{r\leq m}|+|\be^{r\leq m}|=|\ga^{r\leq m}|$, nodes in lower rows must have content $>m$. Thus $t$ is obtained by joining two semistandard tableaux $s_t$ and $u_t$, one of shape $\ga^{r\leq m}\setminus\al^{r\leq m}$ and content $\be^{r\leq m}$ and another of shape $\ga^{r>m}\setminus\al^{r>m}$ and content $\be^{r>m}$ (the second one being shifted down by $m$ rows and with each node's content increased by $m$). Further $s_t$ and $u_t$ satisfy the lattice condition since $t$ does.

Conversely if $u$ and $v$ are semistandard tableaux of shapes $\ga^{r\leq m}\setminus\al^{r\leq m}$ and $\ga^{r>m}\setminus\al^{r>m}$ and contents $\be^{r\leq m}$ and$\be^{r>m}$ respectively, then shifting $v$ down by $m$ rows and increasing the contents of its nodes by $m$ and then taking the union of this new tableau with $u$ we obtain a semistandard tableau $t_{u,v}$ of shape $\ga\setminus \al$ and content $\be$. Since $\be_m\geq\be_{m+1}$, nodes of content $m$ in $t_{u,v}$ are all in rows $\leq m$ while those of content $m+1$ are in rows $>m+1$, $t_{u,v}$ satisfies the lattice condition if both $u$ and $v$ do.

So semistandard tableaux of shape $\ga\setminus\al$ and content $\be$ satisfying the lattice conditions are in bijection with pairs of semistandard tableaux, one of shape $\ga^{r\leq m}\setminus\al^{r\leq m}$ and content $\be^{r\leq m}$ and another of shape $\ga^{r>m}\setminus\al^{r>m}$ and content $\be^{r>m}$, both satifying the lattice condition. The lemma follows.
\end{proof}

\begin{lemma}\label{L101123_5}
Let $\ell,m\geq 1$ and $\al,\be,\ga$ be partitions with $h(\al),h(\ga)\leq \ell$ and $\al_1,\ga_1\leq m$. Then
\[c_{\al,\be}^\ga=c_{\al+(1^\ell),\be}^{\ga+(1^\ell)}=c_{(m)\sqcup\al,\be}^{(m)\sqcup\ga}.\]
\end{lemma}

\begin{proof}
We prove only that $c_{\al,\be}^\ga=c_{\al+(1^\ell),\be}^{\ga+(1^\ell)}$, the proof of $c_{\al,\be}^\ga=c_{(m)\sqcup\al,\be}^{(m)\sqcup\ga}$ being similar (this second equation can otherwise be obtained from the first using \cref{L111223_2}).

We have that $\al\subseteq\ga$ if and only if $\al+(1^\ell)\subseteq\ga+(1^\ell)$. In this case tableaux of shapes $\ga\setminus\al$ and $(\ga+(1^\ell))\setminus(\al+(1^\ell))$ can be naturally identified by simply shifting all nodes by 1 along the rows since $h(\al),h(\ga)\leq \ell$.
\end{proof}

The next results compare certain products of Specht modules and permutation modules. Here, as well as in the following, when considering multiplicities of Specht modules, or other irreducible characteristic 0 modules, in other modules which can be defined in characteristic 0, such multiplicities will be interpreted as multiplicities in characteristic 0.

\begin{thm}\label{L111223_4}
Let $n\geq 0$ and $\nu,\psi\in\Par(n)$. Assume that for some $m,\ell\geq 0$, $h(\nu)=m$ and $|\nu^{r\leq m}|=|\psi^{r\leq 2m}|=\ell$. 
When working over $\C$, $[S^{\nu}\otimes M^{(n-k,k)}:S^\psi]$ is given by
\begin{align*}
&\sum_{j=\max\{0,k+\ell-n\}}^{\min\{k,\ell\}}[S^{\nu^{r\leq m}}\otimes M^{(\ell-j,j)}:S^{\psi^{r\leq 2m}}]\cdot[S^{\nu^{r>m}}\otimes M^{(n-\ell-k+j,k-j)}:S^{\psi^{r>2m}}]
\end{align*}
for every $0\leq k\leq n$.
\end{thm}

\begin{proof}
If $m=0$ then $\ell=0$, so the only term appearing in the sum is that with $j=0$ and then the theorem holds (as in this case $\nu^{r>m}=\nu$ and $\psi^{r>2m}=\psi$). Similarly the theorem holds if $m\geq h(\nu)$ or $2m\geq h(\psi)$ (the two conditions being equivalent by the assumption $|\nu^{r\leq m}|=|\psi^{r\leq 2m}|$).

So we may assume that $m\geq 1$, $m<h(\nu)$ and $2m<h(\psi)$. By \cref{L101123_3},
\[[S^{\nu}\otimes M^{(n-k,k)}:S^\psi]=\sum_{\al\in\Par(n-k),\be\in\Par(k)}c_{\al,\be}^\nu c_{\al,\be}^\psi.\]

Fix $\al\in\Par(n-k)$ and $\be\in\Par(k)$ with $c_{\al,\be}^\nu c_{\al,\be}^\psi>0$. Then, by \cref{L101123_4}, $\al+\be\unrhd\nu$ and $\al\sqcup\be\unlhd\psi$. In particular
\begin{align*}
|\al^{r\leq m}|+|\be^{r^\leq m}|&=|(\al+\be)^{r\leq m}|\geq|\nu^{r\leq m}|=\ell,\\
|\al^{r\leq m}|+|\be^{r\leq m}|&\leq|(\al\sqcup\be)^{r\leq 2m}|\leq|\psi^{r\leq 2m}|=\ell.
\end{align*}
So $|\al^{r\leq m}|+|\be^{r\leq m}|=\ell$ and $(\al\sqcup\be)^{r\leq 2m}=\al^{r\leq m}\sqcup\be^{r\leq m}$, in which case 
$\al\sqcup\be\unlhd\psi$ if and only if both 
$\al^{r\leq m}\sqcup\be^{r\leq m}\unlhd\psi^{r\leq 2m}$ and $\al^{r>m}\sqcup\be^{r>m}\unlhd\psi^{r>2m}$ hold. In particular
\[\al_m,\be_m\geq\psi_{2m}\geq\psi_{2m+1}\geq\al_{m+1},\be_{m+1}.\]
Since $2m<h(\psi)$ we also have by the above equation that $h(\al^{r\leq m}),h(\be^{r\leq m})=m$.

Similarly
\begin{align*}
&[S^{\nu^{r\leq m}}\otimes M^{(\ell-j,j)}:S^{\psi^{r\leq 2m}}]\cdot[S^{\nu^{r>m}}\otimes M^{(n-\ell-k+j,k-j)}:S^{\psi^{r>2m}}]\\
&=\sum_{{\ga\in\Par(\ell-j),\de\in\Par(j),}\atop{\rho\in\Par(n-\ell-k+j),\si\in\Par(k-j)}}c_{\ga,\de}^{\nu^{r\leq m}}c_{\ga,\de}^{\psi^{r\leq 2m}}c_{\rho,\si}^{\nu^{r>m}}c_{\rho,\si}^{\psi^{r>m}}.
\end{align*}
Assume that $c_{\ga,\de}^{\nu^{r\leq m}}c_{\ga,\de}^{\psi^{r\leq 2m}}c_{\rho,\si}^{\nu^{r>m}}c_{\rho,\si}^{\psi^{r>m}}>0$. By \cref{L101123_4} we have that if $c_{\ga,\de}^{\nu^{r\leq m}}>0$ then $h(\ga),h(\de)\leq h(\nu^{r\leq m})=m$, so $h(\ga)+h(\de)\leq 2m= h(\psi^{r\leq 2m})$. From $c_{\ga,\de}^{\psi^{r\leq 2m}}>0$ it follows that $h(\ga)=h(\de)=m$ and $\ga_m,\de_m\geq \psi_{2m}$. From $c_{\rho,\si}^{\psi^{r>2m}}>0$ that $\rho_1,\si_1\leq\psi_{2m+1}$.

Identifying $\al=\ga\sqcup\rho$ and $\be=\de\sqcup\si$ (or equivalently $\ga=\al^{r\leq m}$, $\de=\be^{r\leq m}$, $\rho=\al^{r>m}$, $\si=\be^{r>m}$), it is enough to show that $c_{\al,\be}^\nu=c_{\al^{r\leq m},\be^{r\leq m}}^{\nu^{r\leq m}}c_{\al^{r>m},\be^{r>m}}^{\nu^{r>m}}$ and $c_{\al,\be}^\psi=c_{\al^{r\leq m},\be^{r\leq m}}^{\psi^{r\leq 2m}}c_{\al^{r>m},\be^{r>m}}^{\psi^{r>2m}}$ whenever $\al,\be$ are partitions satisfying $|\al^{r\leq m}|+|\be^{r\leq m}|=\ell$, $\al_m,\be_m\geq\psi_{2m}$ and $\al_{m+1},\be_{m+1}\leq\psi_{2m+1}$.

By \cref{L111223}, $c_{\al,\be}^\nu=c_{\al^{r\leq m},\be^{r\leq m}}^{\nu^{r\leq m}}c_{\al^{r>m},\be^{r>m}}^{\nu^{r>m}}$ holds in this case.

Let $a:=\psi_{2m+1}$. Since $\al_{m+1},\be_{m+1}\leq \psi_{2m+1}=a$ and $\al_m\be_m\geq\psi_m\geq a$, we can decompose $\al=X^\al\sqcup Y^\al\sqcup Z^\al$ with
\begin{align*}
X^\al&=\{(i,j):1\leq i\leq m,\,1\leq j\leq a\},\\
Y^\al&=\{(i,j):(i,j)\in\al,\,i\leq m,\,j>a\},\\
Z^\al&=\{(i,j):(i,j)\in\al,\,i>m,\,j\leq a\}
\end{align*}
and similarly for $\be=X^\be\sqcup Y^\be\sqcup Z^\be$. Using a minor variation of the above decomposition we have that $\psi=X^\psi\sqcup Y^\psi\sqcup Z^\psi$ with
\begin{align*}
X^\psi&=\{(i,j):1\leq i\leq 2m,\,1\leq j\leq a\},\\
Y^\psi&=\{(i,j):(i,j)\in\psi,\,i\leq 2m,\,j>a\},\\
Z^\psi&=\{(i,j):(i,j)\in\psi,\,i>2m,\,j\leq a\}.
\end{align*}
By definition of the above sets
\begin{align*}
\al^{r\leq m}&=X^\al\sqcup Y^\al,&\al^{c\leq a}&=X^\al\sqcup Z^\al,&\al^{r\leq m,c\leq a}&=X^\al,\\
\be^{r\leq m}&=X^\be\sqcup Y^\be,&\be^{c\leq a}&=X^\be\sqcup Z^\be,&\be^{r\leq m,c\leq a}&=X^\be,\\
\psi^{r\leq 2m}&=X^\psi\sqcup Y^\psi,&\psi^{c\leq a}&=X^\psi\sqcup Z^\psi,&\psi^{r\leq 2m,c\leq a}&=X^\psi
\end{align*}
and $\al^{r>m,c>a}=\be^{r>m,c>a}=\psi^{r>2m,c>a}=\varnothing$. 
As $|\al^{r\leq m}|+|\be^{r\leq m}|=|\psi^{r\leq 2m}|$ by assumption and
\[|\al^{r\leq m,c\leq a}|+|\be^{r\leq m,c\leq a}|=|X^\al|+|X^\be|=2am=|X^\psi|=|\psi^{r\leq 2m,c\leq a}|,\]
we also have that $|\al^{c\leq a}|+|\be^{c\leq a}|=|\psi^{c\leq a}|$.

Then
\begin{align*}
&c_{\al,\be}^\psi=c_{\al',\be'}^{\psi'}\\
&=c_{(\al')^{r\leq a},(\be')^{r\leq a}}^{(\psi')^{r\leq a}}c_{(\al')^{r>a},(\be')^{r>a}}^{(\psi')^{r>a}}\\
&=c_{\al^{c\leq a},\be^{c\leq a}}^{\psi^{c\leq a}}c_{\al^{c>a},\be^{c>a}}^{\psi^{c>a}}\\
&=c_{(\al^{r>m,c\leq a}\sqcup(a^m)),(\be^{r>m,c\leq a}\sqcup(a^m))}^{\psi^{r>2m,c\leq a}\sqcup(a^{2m})}c_{\al^{r\leq m,c>a},\be^{r\leq m,c>a}}^{\psi^{r\leq 2m,c>a}}\\
&=c_{\al^{r>m,c\leq a},\be^{r>m,c\leq a}}^{\psi^{r>2m,c\leq a}}c_{\al^{r\leq m,c>a},\be^{r\leq m,c>a}}^{\psi^{r\leq 2m,c>a}}\\
&=c_{\al^{r>m,c\leq a},\be^{r>m,c\leq a}}^{\psi^{r>2m,c\leq a}}c_{(\al')^{c\leq m,r>a},(\be')^{c\leq m,r>a}}^{(\psi')^{c\leq 2m,r>a}}\\
&=c_{\al^{r>m,c\leq a},\be^{r>m,c\leq a}}^{\psi^{r>2m,c\leq a}}c_{(\al')^{c\leq m,r>a},(\be')^{c\leq m,r>a}}^{(\psi')^{c\leq 2m,r>a}}(c_{(m),(m)}^{(2m)})^a\\
&=c_{\al^{r>m,c\leq a},\be^{r>m,c\leq a}}^{\psi^{r>2m,c\leq a}}c_{((\al')^{c\leq m,r>a}\sqcup(m^a)),((\be')^{c\leq m,r>a}\sqcup(m^a))}^{(\psi')^{c\leq 2m,r>a}\sqcup(2m^a)}\\
&=c_{\al^{r>m,c\leq a},\be^{r>m,c\leq a}}^{\psi^{r>2m,c\leq a}}c_{(\al^{r\leq m,c>a}+(a^m)),(\be^{r\leq m,c>a}+(a^m))}^{\psi^{r\leq 2m,c>a}+(a^{2m})}\\
&=c_{\al^{r>m},\be^{r>m}}^{\psi^{r>2m}}c_{\al^{r\leq m},\be^{r\leq m}}^{\psi^{r\leq 2m}}
\end{align*}
with \cref{L111223_2} applied in the first, third, sixth and ninth row, \cref{L111223} on the second and toghether with \cref{L111223_2} repeatedly on the eight row, \cref{L101123_5} repeatedly on the fifth row, using that $c_{(m),(m)}^{(2m)}=1$ by explicit computation on the seventh row and identifying partitions the fourth and tenth row.
\end{proof}

\begin{thm}\label{L181223}
Assume that $n,\ell\geq 0$ and $\al,\be\in\Par(n)$ with $h(\al)\leq\ell$ and $h(\be)\leq 2\ell$. When working over $\C$, for $0\leq k\leq n+2\ell$, 
\[[S^{\al+(2^\ell)}\otimes M^{(n+2\ell-k,k)}:S^{\be+(1^{2\ell})}]=[S^\al\otimes M^{(n+\ell-k,k-\ell)}:S^\be]\]
if $\ell\leq k\leq n+\ell$ and it is
\[[S^{\al+(2^\ell)}\otimes M^{(n+2\ell-k,k)}:S^{\be+(1^{2\ell})}]=0\]
otherwise.
\end{thm}

\begin{proof}
By \cref{L101123_3,L111223_2},
\[[S^\ga\otimes M^{(m-h,h)}:S^\de]=[S^\de\otimes M^{(m-h,h)}:S^\ga]=[S^{\de'}\otimes M^{(m-h,h)}:S^{\ga'}]\]
for each $0\leq h\leq m$ and $\ga,\de\in\Par(m)$. So, by \cref{L111223_4},
\begin{align*}
&[S^{\al+(2^\ell)}\otimes M^{(n+2\ell-k,k)}:S^{\be+(1^{2\ell})}]=[S^{(2\ell)\sqcup\be'}\otimes M^{(n+2\ell-k,k)}:S^{(\ell^2)\sqcup\al'}]\\
&=\sum_{j=\max\{0,k-n\}}^{\min\{k,2\ell\}}[S^{(2\ell)}\otimes M^{(2\ell-j,j)}:S^{(\ell^2)}]\cdot[S^{\be'}\otimes M^{(n-k+j,k-j)}:S^{\al'}]\\
&=\sum_{j=\max\{0,k-n\}}^{\min\{k,2\ell\}}[S^{(2\ell)}\otimes M^{(2\ell-j,j)}:S^{(\ell^2)}]\cdot[S^\al\otimes M^{(n-k+j,k-j)}:S^\be].
\end{align*}
By \cite[Example 17.17]{JamesBook} and definition, $S^{(2\ell)}\cong M^{(2\ell)}\cong 1_{\s_{2\ell}}$. Further, by the same reference, $S^{(\ell^2)}$ is a composition factor with multiplicity 1 of $M^{(\ell^2)}$, but not a composition factor of $M^{(2\ell-j,j)}$ for $j\not=\ell$ (use also that by definition $M^{(2\ell-j,j)}\cong M^{(j,2\ell-j)}$ for $j>\ell$). The theorem then follows.
\end{proof}

\section{Shifted tableaux and tensor products}\label{st}

Tensor products with basic spin representations have been studied in \cite[\S9]{s} in terms of certain analogues of semistandard tableaux satisfying the lattice condition.

A shifted tableau is a map from a set of nodes to the ordered set $\{1'<1<2'<2<\ldots\}$. The content of a shifted tableau $t$ is the composition whose $i$-th part is equal to the number of nodes of $t$ with content $i$ or $i'$. The shape of a tableau is again just the underlying set of nodes. For $t$ a tableau let $\overline{w}(t)$ be the word obtained by concatenating the contents of nodes of $t$, reading row by row, starting from the bottom and within a row from the left. Note that $\overline{w}(t)$ is obtained by reversing the order of the work $w(t)$, as defined in the previous section.

For any word $w=w_1,\ldots,w_\ell$ let $|w|$ be the word obtained by changing any $i'$ to $i$. Further for any $i\geq 1$ and $0\leq j\leq 2\ell$ define $m_i(j)$ to be either the number of $i$ in $w_{\ell-j+1}\ldots w_\ell$ if $0\leq j\leq \ell$ or to be $m_i(\ell)$ plus the number of $i'$ in $w_1\ldots w_{j-\ell}$ if $\ell<j\leq 2\ell$.

We say that a shifted tableau with $\ell$ nodes is semistandard if entries are weakly increasing along rows and columns, no $i'$ is repeated in one row, no $i$ is repeated in one column and  the left-most $i$ in $|\overline{w}(t)|$ is unmarked in $\overline{w}(t)$ whenever $t$ as at least one node of content $i$ or $i'$. We remark that this definition of semistandard shifted tableaux slightly differs from that in \cite{s}, but fits with the counting arguments in \cite[Theorems 8.3, 9.3]{s}. Further we say that $t$ satisfies the lattice condition if both of the following are satisfied:
\begin{enumerate}
\item if $i\geq 2$ and $m_i(j)=m_{i-1}(j)$ and $0\leq j<\ell$ then $\overline{w}(t)_{\ell-j}\not\in\{i,i'\}$,

\item if $i\geq 2$ and $m_i(j)=m_{i-1}(j)$ and $\ell\leq j<2\ell$ then $\overline{w}(t)_{j-\ell+1}\not\in\{i-1,i'\}$.
\end{enumerate}
We define $\c_{\al,\be}$ to be the number of semistandard shifted tableaux of shape $\al$ and content $\be$ which satisfy the lattice condition.


The following lemma allows to study tensor products with basic spin modules in characteristic 0:

\begin{lemma}\label{L101123}
Let $n\geq 1$, $\la\in\Par(n)$ and $\mu\in\Par_2(n)$. When working over $\C$, 
\[[S^\la\otimes S((n)):S(\mu,\eps)]=2^{(h(\mu)-1-a(\mu)+a(n))/2}\c_{\la,\mu}.\]
In particular if $[S^\la\otimes S((n)):S(\mu,\eps)]>0$ then $\mu\unrhd\la$ and if $\la\in\Par_2(n)$ then $[S^\la\otimes S((n)):S(\la,\eps)]>0$.
\end{lemma}

\begin{proof}
By definition, if $(n)$ is even then $S((n))=S((n),0)$ and $a(n)=0$. So in this case the first statement is just \cite[Theorem 9.3]{s}. If $(n)$ is odd then $S((n))=S((n),+)\oplus S((n),-)$ and $a(n)=1$, so that $2^{-a(n)/2}+2^{-a(n)/2}=2^{a(n)/2}$. The first statement then follows from \cite[Theorem 9.3]{s}.

As in any semistandard shifted tableau which satisfies the lattice property nodes in row $i$ have content $j$ or $j'$ for some $j\leq i$, the second statement also holds. Let now $\la\in\Par_2(n)$ and let $t$ be the tableau of shape $\la$ with $t(i,j)=i$ for each node $(i,j)$ of $\la$. Then $t$ is a semistandard shifted tableau of shape and content $\la$ which satisfies the lattice condition, so that the last statement follows from the first.
\end{proof}

The next two lemmas compare numbers of shifted tableaux of certain shapes and contents.

\begin{lemma}\label{L121223}
Let $\al\in\Par(n)$ and $\be\in\Par_2(n)$. If $m\geq 0$ and $|\al^{r\leq m}|=|\be^{r\leq m}|$, then $\c_{\al,\be}=\c_{\al^{r\leq m},\be^{r\leq m}}c_{\al^{r>m},\be^{r>m}}$.
\end{lemma}

\begin{proof}
The lemma is trivial for $m=0$ or $m\geq h(\be)$, since $\c_{\varnothing,\varnothing}=1$. So we may assume that $1\leq m<h(\be)$.

Let $t$ be a semistandard shifted tableaux of shape $\al$ and content $\be$ which satisfies the lattice condition. Then the content of any node in the first $m$ rows is of the form $i$ or $i'$ with $i\leq m$. Since $|\al^{r\leq m}|=|\be^{r\leq m}|$ any node below row $m$ has content $i$ or $i'$ with $i>m$. Thus $t$ can be viewed as the union of a shifted tableau $u_t$ of shape $\al^{r\leq m}$ and content $\be^{r\leq m}$ and a shifted tableau $v_t$  of shape $\al^{r>m}$ and content $\be^{r>m}$ (moved down by $m$ rows and with content increased by $m$). Both $u_t$ and $v_t$ are semistandard and satisfy the lattice condition since $t$ does.

Conversely suppose that $u$ resp. $v$ are a semistandard shifted tableaux of shape $\al^{r\leq m}$ resp. $\al^{r>m}$ and content $\be^{r\leq m}$ resp. $\be^{r>m}$, both of which satisfy the lattice condition. Let $t_{u,v}$ be the union of $u$ and $v$, with $v$ moved down by $m$ rows and with content increased by $m$. Then $t_{u,v}$ is a shifted tableau of shape $\al$ and content $\be$. It follows directly by the definition that $t_{u,v}$ is semistandard. Further any node of content $m$ or $m'$ in $t_{u,v}$ (or in $u$) must be in row $m$. Since the tableaux is semistandard all such nodes are unmarked. Any node of content $m+1$ or $(m+1)'$ in $t_{u,v}$ is in rows $>m$. Since $1\leq m<h(\be)$ and $\be$ is 2-regular, $\be_m>\be_{m+1}$. By these remarks and since $u$ and $v$ satisfy the lattice condition, $t_{u,v}$ also satisfies it.

The lemma then follows by definition of $\c_{\al,\be}$, $\c_{\al^{r\leq m},\be^{r\leq m}}$ and $c_{\al^{r>m},\be^{r>m}}$.
\end{proof}

\begin{lemma}\label{L101123_2}
Let $n,m\geq 1$, $\nu\in\Par(n)$ and $\psi\in\Par_2(n)$ with $h(\nu)\leq m$ and $m-1\leq h(\psi)\leq m$. Then $\c_{\nu,\psi}=\c_{\nu+(1^m),\psi+(1^m)}$.
\end{lemma}

\begin{proof}
Note that in any semistandard shifted tableau satisfying the lattice property any node with content $i$ or $i'$ is in row $j$ with $j\geq i$. Further the top node with content $i$ most appear above the top node with content $i+1$ (if there is a node with content $i+1$).

Assume that $h(\al)=h(\be)$ and $t$ is a semistandard shifted tableau of shape $\al$ and content $\be$ that satisfies the lattice property. Then there exist nodes with content $i$ for each $1\leq i\leq h(\al)$, so the top-most such node must be in row $i$. As row $i$ cannot contain any node of content $j$ or $j'$ with $j>i$, the last node in row $i$ of $t$ has content $i$.

We will use the above properties in the remaining part of the proof without further reference.

Let $t$ be a semistandard shifted tableau of shape $\nu$ and content $\psi$ that satisfies the lattice property. Then we can obtain a shifted tableau $t^+$ of shape $\nu+(1^m)$ and content $\psi+(1^m)$ by adding a node of content $i$ at the end of row $i$ for $1\leq i\leq m$. The added node of content $i$ is the right-most node of content $i$ or $i'$ in $\overline{w}(t^+)$. Since $t$ satisfies the lattice property it can be checked that $t^+$ also satisfies the lattice property. The added nodes are not above any node of $t$ (as $\nu$ is a partition) and have the maximal possible content for nodes in the corresponding rows of shifted tableaux satisfying the lattice property. Further the node corresponding to the left-most $i$ or $i'$ in $\overline{w}(t^+)$ is either the node corresponding to the left-most $i$ in $\overline{w}(t)$ (if $i\leq h(\psi)$) or it is the added node of content $i$ (if $i>h(\psi)$). In either case the content of this node is unmarked and so $t^+$ is semistandard.

Let now $s$ be a semistandard shifted tableau of shape $\nu+(1^m)$ and content $\psi+(1^m)$ that satisfies the lattice property. Then the last node in row $i$ of $s$ has content $i$ for each $1\leq i\leq m$. Removing these nodes we obtain a shifted tableau $s^-$ of shape $\nu$ and content $\psi$. It can be checked that $s^-$ satisfies the lattice property as $s$ does. Further if $1\leq i\leq h(\psi)$ then the node corresponding to the left most $i$ or $i'$ in $\overline{w}(s^-)$ is the node corresponding to the left-most $i$ or $i'$ in $\overline{w}(s)$. In particular its content is unmarked. Since $s$ is semistandard it follows that $s^-$ is also semistandard.

As the maps $t\to t^+$ and $s\to s^-$ are inverses of each other, the lemma follows.
\end{proof}

The next two theorems, which are the last two results needed to prove \cref{t1,t2}, compare certain linear combinations involving tensor products with basic spin modules in characteristic 0.

\begin{thm}\label{L151223}
Let $1\leq\ell<n$ and $a,m\geq 1$. Assume that the ground field is $\C$ and that for some coefficients $e_\nu,f_\psi,g_\psi,E_\pi,F_\rho,G_\rho$
\begin{align*}
\sum_{{\nu\in\Par(\ell):}\atop{h(\nu)=m,\nu_m\geq a}}e_\nu[S^\nu\otimes S((\ell))]&=\sum_{{\psi\in\Par_2(\ell):}\atop{h(\psi)=m,\psi_m\geq a}}f_\psi[S(\psi)]+\sum_{{\psi\in\Par_2(\ell):}\atop{h(\psi)=m,\psi_m<a}}g_\psi[S(\psi)],\\
\sum_{{\pi\in\Par(n-\ell):}\atop{\pi_1<a}}E_\pi[S^\pi\otimes S((n-\ell))]&=\sum_{{\rho\in\Par_2(n-\ell):}\atop{\rho_1<a}}F_\rho[S(\rho)]+\sum_{{\rho\in\Par_2(n-\ell):}\atop{\rho_1\geq a}}G_\rho[S(\rho)].
\end{align*}
Then, for some coefficients $h_\si$,
\begin{align*}
&\sum_{{\nu\in\Par(\ell):}\atop{h(\nu)=m,\nu_m\geq a}}\sum_{{\pi\in\Par(n-\ell):}\atop{\pi_1<a}}e_\nu E_\pi[S^{\nu\sqcup\pi}\otimes S((n))]\\
&=\sum_{{\si\in\Par_2(n):|\si^{r\leq m}|=\ell,}\atop{\si_m\geq a,\si_{m+1}<a}}2^{a(\si^{r\leq m})a(\si^{r>m})+(1-a(\ell))(1-a(n-\ell))}f_{\si^{r\leq m}}F_{\si^{r>m}}[S(\si)]\\
&\hspace{11pt}+\sum_{{{\si\in\Par_2(n):|\si^{r\leq m}|>\ell}\atop{\text{or }|\si^{r\leq m}|=\ell,\si_m<a}}\atop{\text{or }|\si^{r\leq m}|=\ell,\si_{m+1}\geq a}}h_\si[S(\si)].
\end{align*}
\end{thm}

\begin{proof}
By \cref{L101123} we may assume that $\si\in\Par_2(n)$ with $|\si^{r\leq m}|=\ell$, $\si_m\geq a$ and $\si_{m+1}<a$ and find the multiplicity of $S(\si,\eps)$ in
\[\sum_{{\nu\in\Par(\ell):}\atop{h(\nu)=m,\nu_m\geq a}}\sum_{{\pi\in\Par(n-\ell):}\atop{\pi_1<a}}e_\nu E_\pi[S^{\nu\sqcup\pi}\otimes S((n))].\]
Since $1\leq \ell<n$, by \cref{L101123},
\begin{align*}
f_{\si^{r\leq m}}&=\sum_{{\nu\in\Par(\ell):}\atop{h(\nu)=m,\nu_m\geq a}}e_\nu[S^\nu\otimes S((\ell)):S(\si^{r\leq m},\de)]\\
&=\sum_{{\nu\in\Par(\ell):}\atop{h(\nu)=m,\nu_m\geq a}}2^{(h(\si^{r\leq m})-1-a(\si^{r\leq m})+a(\ell))/2}e_\nu\c_{\si^{r\leq m},\nu},\\
F_{\si^{r>m}}&=\sum_{{\pi\in\Par(n-\ell):}\atop{\pi_1<a}}E_\pi[S^\pi\otimes S((n-\ell)):S(\si^{r>m},\de')]\\
&=\sum_{{\pi\in\Par(n-\ell):}\atop{\pi_1<a}}2^{(h(\si^{r>m})-1-a(\si^{r>m}))/2+a(n-\ell)}E_\pi\c_{\si^{r>m},\pi}.
\end{align*}
Using also \cref{L121223} we further have that
\begin{align*}
&\sum_{{\nu\in\Par(\ell):}\atop{h(\nu)=m,\nu_m\geq a}}\sum_{{\pi\in\Par(n-\ell):}\atop{\pi_1<a}}e_\nu E_\pi[S^{\nu\sqcup\pi}\otimes S((n)):S(\si,\eps)]\\
&=\sum_{{\nu\in\Par(\ell):}\atop{h(\nu)=m,\nu_m\geq a}}\sum_{{\pi\in\Par(n-\ell):}\atop{\pi_1<a}}2^{(h(\si)-1-a(\si)+a(n))/2}e_\nu E_\pi \c_{\si,\nu\sqcup\pi}\\
&=\sum_{{\nu\in\Par(\ell):}\atop{h(\nu)=m,\nu_m\geq a}}\sum_{{\pi\in\Par(n-\ell):}\atop{\pi_1<a}}2^{(h(\si)-1-a(\si)+a(n))/2}e_\nu E_\pi \c_{\si^{r\leq m},\nu}\c_{\si^{r>m},\pi}\\
&=2^{(1-a(\si)+a(\si^{r\leq m})+a(\si^{r>m})+a(n)-a(\ell)-a(n-\ell))/2}\\
&\hspace{11pt}\left(\sum_{{\nu\in\Par(\ell):}\atop{h(\nu)=m,\nu_m\geq a}}2^{(h(\si^{r\leq m})-1-a(\si^{r\leq m})+a(\ell))/2}e_\nu\c_{\si^{r\leq m},\nu}\right)\\
&\hspace{11pt}\left(\sum_{{\pi\in\Par(n-\ell):}\atop{\pi_1<a}}2^{(h(\si^{r>m})-1-a(\si^{r>m}))/2+a(n-\ell)}E_\pi\c_{\si^{r>m},\pi}\right)\\
&=2^{(1-a(\si)+a(\si^{r\leq m})+a(\si^{r>m})+a(n)-a(\ell)-a(n-\ell))/2}f_{\si^{r\leq m}}F_{\si^{r>m}}.
\end{align*}
The theorem then follows from $a(\si)=a(\si^{r\leq m})+a(\si^{r>m})-2a(\si^{r\leq m})a(\si^{r>m})$ and $a(n)=1-a(\ell)-a(n-\ell)+2a(\ell)a(n-\ell)$ when $1\leq \ell<n$ (these formulas can be checked by considering all combinations of the appearing $a(\al)$).
\end{proof}

\begin{thm}\label{L131123_3}
Let $n,\ell\geq 1$. Assume that the ground field is $\C$ and that for some coefficients $d_\la$, $e_\la$ and $f_\la$
\[\sum_{\la\in\Par(n):h(\la)=\ell}d_\la[S^\la\otimes S((n))]=\sum_{\la\in\Par_2(n):h(\la)=\ell}e_\la[S(\la)]+\sum_{\la\in\Par_2(n):h(\la)<\ell}f_\la[S(\la)].\]
Then, for some coefficients $g_\mu$,
\begin{align*}
\sum_{\la\in\Par(n):h(\la)=\ell}d_\la[S^{\la+(2^\ell)}\otimes S((n))]&=\sum_{\la\in\Par_2(n):h(\la)=\ell}e_\la[S(\la+(2^\ell))]\\
&\hspace{11pt}+\sum_{{\mu\in\Par_2(n+2\ell):h(\mu)<\ell}\atop{\text{or }h(\mu)=\ell\text{ and }\mu_\ell\leq 2}}g_\mu[S(\mu)].
\end{align*}
\end{thm}

\begin{proof}
Let $\mu\in\Par_2(n+2\ell)$. If $h(\mu)>\ell$ then $S(\mu,\eps)$ does not appear in $\sum_{\la\in\Par(n):h(\la)=\ell}d_\la[S^{\la+(2^\ell,1^m)}\otimes S((n))]$ by \cref{L101123}. So we may assume that $\mu=\nu+(2^\ell)$ for some $\nu\in\Par_2(n)$ with $h(\nu)=\ell$. By \cref{L101123},
\begin{align*}
&\sum_{\la\in\Par(n):h(\la)=\ell}d_\la[S^\la\otimes S((n)):S(\nu,\eps')]\\
&=2^{(h(\nu)-1-a(\nu)+a(n))/2}\sum_{\la\in\Par(n):h(\la)=\ell}d_\la\c_{\la,\nu},
\end{align*}
and
\begin{align*}
&\sum_{\la\in\Par(n):h(\la)=\ell}d_\la[S^{\la+(2^\ell)}\otimes S((n+2\ell)):S(\nu+(2^\ell),\eps)]\\
&=2^{(h(\nu+(2^\ell))-1-a(\nu+(2^\ell))+a(n+2\ell))/2}\sum_{\la\in\Par(n):h(\la)=\ell}d_\la\c_{\la+(2^\ell),\nu+(2^\ell)}.
\end{align*}
By \cref{L101123_2}, $\c_{\la,\nu}=\c_{\la+(2^\ell),\nu+(2^\ell)}$. Further $h(\nu)=\ell=h(\nu+(2^\ell))$, from which it also follows that $a(\nu)=a(\nu+(2^ell))$, and since $n\geq 1$, $a(n)=a(n+2\ell)$.

So the multiplicity of $S(\nu,\eps')$ in
\[\sum_{\la\in\Par(n):h(\la)=\ell}d_\la[S^\la\otimes S((n))]\]
is equal to the multiplicity of $S(\nu+(2^\ell),\eps)$ in
\[\sum_{\la\in\Par(n):h(\la)=\ell}d_\la[S^{\la+(2^\ell)}\otimes S((n))].\]
The theorem follows.
\end{proof}

\section{Proof of \cref{t1}}\label{st1}

Let $\la,\mu\in\Par_2(n)$ and $\ell,m\geq 0$ such that $|\la^{r\leq m}|=\ell=|\mu^{r\leq 2m}|$. If $\ell\in\{0,n\}$ the theorem is trivial since $S(\varnothing,0)\cong 1_{\widetilde{\s}_0}\cong D^\varnothing$. So for the rest of the proof we will assume that $1\leq \ell\leq n-1$. In this case $\la_m>\la_{m+1}\geq 1$. 

By \cref{regs} if $[S(\la):D^\mu]\not=0$ then $\mu\unrhd\dbl(\la)$. Similarly if
\[[S(\la^{r\leq m}):D^{\mu^{r\leq 2m}}]\cdot[S(\la^{r>m}):D^{\mu^{r>2m}}]\not=0\]
then $\mu^{r\leq 2m}\unrhd\dbl(\la^{r\leq m})$ and $\mu^{r>2m}\unrhd\dbl(\la^{r>m})$.

In particular, since $|\la^{r\leq m}|=|\mu^{r\leq 2m}|$, if either of $\la^{r\leq m}$ or $\mu^{r\leq 2m}$ is non-empty, then $\la_m\geq\mu_{2m-1}+\mu_{2m}$ and $\la_{m+1}\leq\mu_{2m+1}+\mu_{2m+2}$.

In view of \cref{L101123} there exist coefficients $e_\nu$, $E_\pi$, $f_\psi$ and $F_\phi$ with
\begin{align*}
\sum_{{\nu\in\Par_2(\ell):h(\nu)=m,}\atop{\nu_m\geq\mu_{2m-1}+\mu_{2m}}}e_\nu[S^\nu\otimes S((\ell))]&=[S(\la^{r\leq m})]+\sum_{{\psi\in\Par_2(\ell):h(\nu)\leq m,}\atop{\psi_m<\mu_{2m-1}+\mu_{2m}}}f_\psi[S(\psi)],\\
\sum_{{\pi\in\Par_2(n-\ell):}\atop{\pi_1\leq\mu_{2m+1}+\mu_{2m+2}}}E_\pi[S^\pi\otimes S((n-\ell))]&=[S(\la^{r>m})]+\sum_{{\phi\in\Par_2(n-\ell):}\atop{\phi_1>\mu_{2m+1}+\mu_{2m+2}}}F_\psi[S(\psi)].
\end{align*}
Then, from \cref{L121223}, for some coefficients $g_\ga$,
\begin{align*}
&\sum_{{\nu\in\Par_2(\ell):h(\nu)=m,}\atop{\nu_m\geq \mu_{2m-1}+\mu_{2m}}}\sum_{{\pi\in\Par_2(n-\ell):}\atop{\pi_1\leq\mu_{2m+1}+\mu_{2m+2}}}e_\nu E_\pi[S^{\nu\sqcup\pi}\otimes S((n))]\\
&=2^{a(\si^{r\leq m})a(\si^{r>m})+(1-a(\ell))(1-a(n-\ell))}[S(\la)]+\sum_{{{\ga\in\Par_2(n):|\ga^{r\leq m}|>\ell,}\atop{\text{or }|\ga^{r\leq m}|=\ell,\ga_m<\mu_{2m-1}+\mu_{2m}}}\atop{\text{or }|\ga^{r\leq m}|=\ell,\ga_{m+1}>\mu_{2m+1}+\mu_{2m+2}}}g_\ga[S(\ga)].
\end{align*}

By \cref{regs} no term appearing in the sums in the right-hand sides of the above equations has composition factors of the form $D^{\mu^{r\leq 2m}}$, $D^{\mu^{r>2m}}$ or $D^\mu$. So
\begin{align*}
[S(\la^{r\leq m}):D^{\mu^{\leq 2m}}]&=\sum_{{\nu\in\Par_2(\ell):h(\nu)=m,}\atop{\nu_m\geq\mu_{2m-1}+\mu_{2m}}}e_\nu[S^\nu\otimes S((\ell)):D^{\mu^{\leq 2m}}],\\
[S(\la^{r>m}):D^{\mu^{r>2m}}]&=\sum_{{\pi\in\Par_2(n-\ell):}\atop{\pi_1\leq\mu_{2m+1}+\mu_{2m+2}}}E_\pi[S^\pi\otimes S((n-\ell)):D^{\mu^{r>2m}}],\\
2^a[S(\la):D^\mu]&=\sum_{{\nu\in\Par_2(\ell):h(\nu)=m,}\atop{\nu_m\geq \mu_{2m-1}+\mu_{2m}}}\sum_{{\pi\in\Par_2(n-\ell):}\atop{\pi_1\leq\mu_{2m+1}+\mu_{2m+2}}}e_\nu E_\pi[S^{\nu\sqcup\pi}\otimes S((n)):D^\mu]
\end{align*}
with $a=a(\si^{r\leq m})a(\si^{r>m})+(1-a(\ell))(1-a(n-\ell))$ and then, in view of \cref{bs},
\begin{align*}
&2^{-a(\ell)-a(n-\ell)}[S(\la^{r\leq m}):D^{\mu^{\leq 2m}}]\cdot[S(\la^{r>m}):D^{\mu^{r>2m}}]\\
&=\sum_{{\nu\in\Par_2(\ell):h(\nu)=m,}\atop{\nu_m\geq\mu_{2m-1}+\mu_{2m}}}\sum_{{\pi\in\Par_2(n-\ell):}\atop{\pi_1\leq\mu_{2m+1}+\mu_{2m+2}}}\Bigl(e_\nu E_\pi\\
&\hspace{60pt}[S^\nu\otimes D^{\dbl(\ell)}:D^{\mu^{\leq 2m}}]\cdot[S^\pi\otimes D^{\dbl(n-\ell)}:D^{\mu^{r>2m}}]\Bigr)\\
&=\sum_{{\nu\in\Par_2(\ell):h(\nu)=m,}\atop{\nu_m\geq\mu_{2m-1}+\mu_{2m}}}\sum_{{\pi\in\Par_2(n-\ell):}\atop{\pi_1\leq\mu_{2m+1}+\mu_{2m+2}}}\Bigl(e_\nu E_\pi\\
&\hspace{60pt}[(S^\nu\boxtimes S^\pi)\otimes (D^{\dbl(\ell)}\boxtimes D^{\dbl(n-\ell)}):D^{\mu^{r\leq 2m}}\boxtimes D^{\mu^{r>2m}}]\Bigr)
\end{align*}
and
\begin{align*}
&2^{a(\si^{r\leq m})a(\si^{r>m})+(1-a(\ell))(1-a(n-\ell))-a(n)}[S(\la):D^\mu]\\
&=\sum_{{\nu\in\Par_2(\ell):h(\nu)=m,}\atop{\nu_m\geq \mu_{2m-1}+\mu_{2m}}}\sum_{{\pi\in\Par_2(n-\ell):}\atop{\pi_1\leq\mu_{2m+1}+\mu_{2m+2}}}e_\nu E_\pi[S^{\nu\sqcup\pi}\otimes D^{\dbl(n)}:D^\mu].
\end{align*}

By \cref{bsm} there exist $t_{n,h}$ for $0\leq h\leq n$ with
\[[D^{\dbl(n)}]=\sum_{h=0}^{n}t_{n,h}[M^{(n-h,h)}].\]
By \cref{L111223_3}, for the same coefficients $t_{n,h}$,
\begin{align*}
&2^{1-(1-a(\ell))(1-a(n-\ell))}[D^{\dbl(\ell)}\boxtimes D^{\dbl(n-\ell)}]\\
&=\sum_{h=0}^{n}\sum_{j=\max\{0,h+\ell-n\}}^{\min\{h,\ell\}}t_{n,h}
[M^{(\ell-j,j)}\boxtimes M^{(n-\ell-h+j,h-j)}].
\end{align*}

In particular
\begin{align*}
&2^{a(\si^{r\leq m})+a(\si^{r>m})-a(\ell)-a(n-\ell)+1-(1-a(\ell))(1-a(n-\ell))}\\
&\hspace{50pt}[S(\la^{r\leq m},\eps):D^{\mu^{\leq 2m}}]\cdot[S(\la^{r>m},\eps'):D^{\mu^{r>2m}}]\\
&=2^{-a(\ell)-a(n-\ell)-1+(1-a(\ell))(1-a(n-\ell))}[S(\la^{r\leq m}):D^{\mu^{\leq 2m}}]\cdot[S(\la^{r>m}):D^{\mu^{r>2m}}]\\
&=\sum_{{\nu\in\Par_2(\ell):h(\nu)=m,}\atop{\nu_m\geq\mu_{2m-1}+\mu_{2m}}}\sum_{{\pi\in\Par_2(n-\ell):}\atop{\pi_1\leq\mu_{2m+1}+\mu_{2m+2}}}\sum_{h=0}^{n}\sum_{j=\max\{0,h+\ell-n\}}^{\min\{h,\ell\}} \Bigl(e_\nu E_\pi t_{n,h}\\
&\hspace{40pt}[(S^\nu\boxtimes S^\pi)\otimes (M^{(\ell-j,j)}\boxtimes M^{(n-\ell-h+j,h-j)}):D^{\mu^{r\leq 2m}}\boxtimes D^{\mu^{r>2m}}]\Bigr)\\
&=\sum_{{\nu\in\Par_2(\ell):h(\nu)=m,}\atop{\nu_m\geq\mu_{2m-1}+\mu_{2m}}}\sum_{{\pi\in\Par_2(n-\ell):}\atop{\pi_1\leq\mu_{2m+1}+\mu_{2m+2}}}\sum_{h=0}^{n}\sum_{j=\max\{0,h+\ell-n\}}^{\min\{h,\ell\}} \Bigl(e_\nu E_\pi t_{n,h}\\
&\hspace{40pt}[S^\nu\otimes M^{(\ell-j,j)}:D^{\mu^{r\leq 2m}}]\cdot [S^\pi\otimes M^{(n-\ell-h+j,h-j)}:D^{\mu^{r>2m}}]\Bigr)
\end{align*}
and
\begin{align*}
&2^{a(\si)+a(\si^{r\leq m})a(\si^{r>m})+(1-a(\ell))(1-a(n-\ell))-a(n)}[S(\la,\de):D^\mu]\\
&=2^{a(\si^{r\leq m})a(\si^{r>m})+(1-a(\ell))(1-a(n-\ell))-a(n)}[S(\la):D^\mu]\\
&=\sum_{{\nu\in\Par_2(\ell):h(\nu)=m,}\atop{\nu_m\geq \mu_{2m-1}+\mu_{2m}}}\sum_{{\pi\in\Par_2(n-\ell):}\atop{\pi_1\leq\mu_{2m+1}+\mu_{2m+2}}}\sum_{h=0}^{n}e_\nu E_\pi t_{n,h}[S^{\nu\sqcup\pi}\otimes M^{(n-h,h)}:D^\mu].
\end{align*}
Since $a(\si)=a(\si^{r\leq m})+a(\si^{r>m})-2a(\si^{r\leq m})a(\si^{r>m})$ and $a(n)=1-a(\ell)-a(n-\ell)+2a(\ell)a(n-\ell)$, we have that
\begin{align*}
&a(\si^{r\leq m})a(\si^{r>m})+(a(\si)+a(\si^{r\leq m})a(\si^{r>m})+(1-a(\ell))(1-a(n-\ell))-a(n))\\
&=a(\si^{r\leq m})+a(\si^{r>m})-a(\ell)-a(n-\ell)+1-(1-a(\ell))(1-a(n-\ell)).
\end{align*}
In view of this last equation, it is enough to prove that
\begin{align*}
&[S^{\nu\sqcup\pi}\otimes M^{(n-h,h)}:D^\mu]\\
&=\sum_{j=\max\{0,h+\ell-n\}}^{\min\{h,\ell\}}[S^\nu\otimes M^{(\ell-j,j)}:D^{\mu^{r\leq 2m}}]\cdot [S^\pi\otimes M^{(n-\ell-h+j,h-j)}:D^{\mu^{r>2m}}]
\end{align*}
for every $\nu\in\Par_2(\ell)$ with $h(\nu)=m$ and $\nu_m\geq\mu_{2m-1}+\mu_{2m}$, every $\pi\in\Par_2(n-\ell)$ with $\pi_1\leq\mu_{2m+1}+\mu_{2m+2}$ and every $0\leq h\leq n$.

So let such $\nu$, $\pi$ and $h$ be fixed. The above equation can be rewritten as
\begin{align*}
&\sum_{\psi\in\Par(n)}[S^{\nu\sqcup\pi}\otimes M^{(n-h,h)}:S^\psi]\cdot[S^\psi:D^\mu]\\
&=\sum_{\rho\in\Par(\ell)}\sum_{\si\in\Par(n-\ell)}\sum_{j=\max\{0,h+\ell-n\}}^{\min\{h,\ell\}}\Bigl([S^\nu\otimes M^{(\ell-j,j)}:S^{\rho}]\cdot[S^\rho:D^{\mu^{r\leq 2m}}]\\
&\hspace{60pt}[S^\pi\otimes M^{(n-\ell-h+j,h-j)}:S^\si]\cdot[S^\si:D^{\mu^{r>2m}}]\Bigr)
\end{align*}

By \cref{L101123_3,L131123} if $[S^{\nu\sqcup\pi}\otimes M^{(n-h,h)}:S^\psi]=0$ then $|\psi^{r\leq 2m}|<|(\nu\sqcup\pi)^{r\leq m}|=|\nu|=\ell$. By \cref{reg}, $[S^\psi:D^\mu]=0$ if $|\psi^{r\leq 2m}|>\ell=|\mu^{r\leq 2m}|$. By the same result, $[S^\psi:D^\mu]=0$ also if $|\psi^{r\leq 2m}|=\ell$ and at least one of $\psi_{2m}<\mu_{2m}$ or $\psi_{2m+1}>\mu_{2m+1}$ holds.

Similarly $[S^\nu\otimes M^{(\ell-j,j)}:S^{\rho}]=0$ if $h(\rho)>2h(\nu)=2m$, $[S^\rho:D^{\mu^{r>2m}}]=0$ if $\rho_{2m}<\mu_{2m}$ (in particular if $h(\rho)<2m$) and $[S^\si:D^{\mu^{r>2m}}]=0$ if $\si_1>\mu_{2m+1}$.

So it is enough to check that
\begin{align*}
&\sum_{{\psi\in\Par(n):|\psi^{r\leq 2m}|=\ell}\atop{\psi_{2m}\geq\mu_{2m},\psi_{2m+1}\leq\mu_{2m+1}}}[S^{\nu\sqcup\pi}\otimes M^{(n-h,h)}:S^\psi]\cdot[S^\psi:D^\mu]\\
&=\sum_{{\rho\in\Par(\ell):}\atop{h(\rho)=2m,\rho_{2m}\geq \mu_{2m}}}\sum_{{\si\in\Par(n-\ell):}\atop{\si_1\leq\mu_{2m+1}}}\sum_{j=\max\{0,h+\ell-n\}}^{\min\{h,\ell\}}\Bigl([S^\nu\otimes M^{(\ell-j,j)}:S^{\rho}]\cdot[S^\rho:D^{\mu^{r\leq 2m}}]\\
&\hspace{60pt} [S^\pi\otimes M^{(n-\ell-h+j,h-j)}:S^\si]\cdot[S^\si:D^{\mu^{r>2m}}]\Bigr),
\end{align*}
which holds by \cref{T1,L111223_4} (identifying pairs $(\rho,\si)$ as above with $(\psi^{r\leq 2m},\psi^{r>2m})$).

\section{Proof of \cref{t2}}\label{st2}

Let $\ell:=|\mu^{c>m}|=|\la^{c>2m}|$. Since $|\mu^{c>m}|=|\la^{c>2m}|$, the assumption $\mu\unrhd\bdbl(\la)$ is equivalent to $\mu^{c\leq m}\unrhd\bdbl(\la^{c\leq 2m})$ and $\mu^{c>m}\unrhd\bdbl(\la^{c>2m})$ (holding simultaneously).

By definition $b=\la'_{2m}$, so that $2b\leq\bdbl(\la^{c\leq 2m})'_m\leq 2b+1$. As $\mu^{c\leq m}$ has $m$ columns and $\mu^{c\leq m}\unrhd\bdbl(\la^{c\leq 2m})$, it follows that
\[\mu'_m\geq \bdbl(\la^{c\leq 2m})'_m\geq 2b\]
and then also that $|\mu^{r>2b,c\leq m}|=n-\ell-2bm=|\la^{r>b,c\leq 2m}|$.

From $\mu^{c>m}\unrhd\bdbl(\la^{c>2m})$ it follows that
\[\mu'_{m+1}=h(\mu^{c>m})\leq h(\bdbl(\la^{c>2m}))\leq 2\la'_{2m+1}\leq 2\la'_{2m}\leq 2b.\]

Since $\la,\mu\in\Par_2(n)$ it follows that $2b-1\leq \mu_{m+1}'\leq 2b$ and $2b\leq\mu_m'\leq 2b+1$. As $\la_{2m}'=b$ by definition, $\la'_{2m+1}\leq b$ and
\begin{align*}
&\la^{r>b,c\leq 2m}=\la^{r>b},&&\mu^{r>2b,c\leq m}=\mu^{r>2b},\\
&\la^{c>2m}+(2m^b)=\la^{r\leq b},&&\mu^{c>m}+(m^{2b})=\mu^{r\leq 2b}.
\end{align*}
In particular $|\mu^{r\leq 2b}|=|\la^{r\leq b}|$ and then by \cref{t1},
\begin{align*}
&[S(\la,\de):D^{\mu}]=2^{a(\la^{r\leq b})a(\la^{r>b})}[S(\la^{r\leq b},\eps''):D^{\mu^{r\leq 2b}}]\cdot[S(\la^{r>b},\eps'):D^{\mu^{r>2b}}]\\
&=2^{a(\la^{r\leq b})a(\la^{r>b,c\leq 2m})}[S(\la^{r\leq b},\eps''):D^{\mu^{r\leq 2b}}]\cdot[S(\la^{r>b,c\leq 2m},\eps'):D^{\mu^{r>2b,c\leq m}}].
\end{align*}

Since $\mu^{c\leq m}\unrhd\bdbl(\la^{c\leq 2m})$, we have that $\la_{2m+1}'\geq\mu_{m+1}'/2\geq (2b-1)/2$. Further $\la_{2m+1}'\leq b$ by definition of $b$. So $\la_{2m+1}'=b$ and then $a(\la^{r\leq b})=a(\la^{c>2m})$ (that is $\eps''=\pm\eps$) as $\la^{c>2m}+(2m^b)=\la^{r\leq b}$. Thus it is enough to prove that
\[[S(\la^{r\leq b},\eps):D^{\mu^{r\leq 2b}}]=[S(\la^{c>2m},\eps):D^{\mu^{c>m}}]\]
or equivalently (in view of \cref{L151123_2}) that
\[[S(\la^{r\leq b}):D^{\mu^{r\leq 2b}}]=[S(\la^{c>2m}):D^{\mu^{c>m}}].\]

By the previous part of \cref{st2}, $\la_{2m+1}'=b$, $2b-1\leq \mu_{m+1}'\leq 2b$, $\la^{c>2m}+(2m^b)=\la^{r\leq b}$ and $\mu^{c> m}+(m^{2b})=\mu^{r\leq 2b}$.

It is thus enough to prove that
\[[S(\nu):D^{\pi}]=[S(\nu+(2^b)):D^{\pi+(1^{2b})}]\]
whenever $\nu,\pi\in\Par_2(N)$ for some $N$ with $h(\nu)=b$ and $2b-1\leq h(\pi)\leq 2b$, since then repeated application of this starting from $\la^{c>2m}$ and $\mu^{c>m}$ yields the results. We may assume that $b\geq 1$, in which case $N\geq 1$ (as $N=|\nu|$ and $h(\nu)=b\geq 1$). 

By \cref{bsm} there exist coefficient $t_h$ with
\begin{align*}
[D^{\dbl(N)}]&=\sum_{h=0}^{\lfloor(N-1)/2\rfloor}t_h[M^{(\lceil(N+1)/2+h,\lfloor(N-1)/2\rfloor-h)}],\\
[D^{\dbl(N+2b)}]&=\sum_{h=0}^{\lfloor(N-1)/2\rfloor+b}t_h[M^{(\lceil(N+1)/2+b+h,\lfloor(N-1)/2\rfloor+b-h)}]
\end{align*}
and by \cref{L131123_3} there exist coefficients $d_\al$, $f_\be$, $g_\ga$ with
\begin{align*}
\sum_{{\al\in\Par(N):}\atop{h(\al)=b}}d_\al[S^\al\otimes S((N))]&=[S(\nu)]+\sum_{{\be\in\Par_2(N):}\atop{h(\be)<b}}f_\be[S(\be)],\\
\sum_{{\al\in\Par(N):}\atop{h(\al)=b}}d_\al[S^{\al+(2^b)}\otimes S((N+2b))]&=[S(\nu+(2^b))]+\sum_{{\ga\in\Par_2(N+2b):h(\ga)<b}\atop{\text{or }h(\ga)=b\text{ and }\ga_b\leq 2}}g_\ga[S(\ga)].
\end{align*}

So, by \cref{regs,bs},
\begin{align*}
&[S(\nu):D^\pi]=\sum_{{\al\in\Par(N):}\atop{h(\al)=b}}d_\al[S^\al\otimes S((N)):D^\pi]\\
&=\sum_{{\al\in\Par(N):}\atop{h(\al)=b}}\sum_{h=0}^{\lfloor(N-1)/2\rfloor}d_\al t_h2^{a(N)}[S^\al\otimes M^{(\lceil(N+1)/2+h,\lfloor(N-1)/2\rfloor-h)}:D^\pi]\\
&=\sum_{{\al\in\Par(N):}\atop{h(\al)=b}}\sum_{h=0}^{\lfloor(N-1)/2\rfloor}\sum_{\psi\in\Par(N)}\Bigl(d_\al t_h2^{a(N)}\\
&\hspace{60pt}[S^\al\otimes M^{(\lceil(N+1)/2+h,\lfloor(N-1)/2\rfloor-h)}:S^\psi]\cdot[S^\psi:D^\pi]\Bigr).
\end{align*}
By \cref{L111223_2,L111223}, $[S^\al\otimes M^{(N-k,k)}:S^\psi]=0$ for every $0\leq k\leq N$ if $h(\psi)>2h(\al)=2b$. So
\begin{align*}
[S(\nu):D^\pi]&=\sum_{{\al\in\Par(N):}\atop{h(\al)=b}}\sum_{h=0}^{\lfloor(N-1)/2\rfloor}\sum_{{\psi\in\Par(N):}\atop{h(\psi)\leq 2b}}\Bigl(d_\al t_h2^{a(N)}\\
&\hspace{60pt}[S^\al\otimes M^{(\lceil(N+1)/2+h,\lfloor(N-1)/2\rfloor-h)}:S^\psi]\cdot[S^\psi:D^\pi]\Bigr).
\end{align*}
Similarly
\begin{align*}
&[S(\nu+(2^b)):D^{\pi+(1^{2b})}]=\sum_{{\al\in\Par(N):}\atop{h(\al)=b}}\sum_{h=0}^{\lfloor(N-1)/2\rfloor+b}\sum_{{\psi\in\Par(N):}\atop{h(\psi)\leq 2b}}\Bigl(d_\al t_h2^{a(N+2b)}\\
&\hspace{10pt}[S^{\al+(2^b)}\otimes M^{(\lceil(N+1)/2+b+h,\lfloor(N-1)/2\rfloor+b-h)}:S^{\psi+(1^{2b})}]\cdot[S^{\psi+(1^{2b})}:D^{\pi+(1^{2b})}]\Bigr).
\end{align*}

By \cref{t2}, $[S^{\psi+(1^{2b})}:D^{\pi+(1^{2b})}]=[S^{\psi}:D^{\pi}]$. Further by \cref{L181223},
\begin{align*}
&[S^{\al+(2^b)}\otimes M^{(\lceil(N+1)/2+b+h,\lfloor(N-1)/2\rfloor+b-h)}:S^{\psi+(1^{2b})}]\\
&=[S^\al\otimes M^{(\lceil(N+1)/2+h,\lfloor(N-1)/2\rfloor-h)}:S^\psi]
\end{align*}
if $0\leq h\leq \lfloor(N-1)/2$ and it is $0$ otherwise. By definition, $a(N)=a(N+2b)$ since $N\geq 1$. The claim follows.


\begin{thebibliography}{99}

\bibitem{ben} D. Benson, Spin modules for symmetric groups, {\em J. London Math. Soc.} {\bf 38} (1988), 250--262.

\bibitem{bo} C. Bessenrodt, J.B. Olsson, The 2-blocks of the covering groups of the symmetric groups, {\em Adv. Math.} {\bf 129} (1997), 261--300.

\bibitem{d} S. Donkin, A note on decomposition numbers for general linear groups and symmetric groups, {\em Math. Proc. Cambridge Philos. Soc.} {\bf 97} (1985), 57--62.

\bibitem{JamesBook} G. D. James, {\em The Representation Theory of the Symmetric Groups}, Lecture Notes in Mathematics, Vol. {\bf 682}, Springer, NewYork/Heidelberg/Berlin, 1978.

\bibitem{j3} G.D. James, On the decomposition matrices of the symmetric groups. III, {\em J. Algebra} {\bf 71} (1981), 115--122.

\bibitem{KBook} A.S. Kleshchev, {\em Linear and Projective Representations of Symmetric Groups}, Cambridge University Press, Cambridge, 2005.

\bibitem{m3} L. Morotti, Composition factors of 2-parts spin representations of symmetric groups in characteristic 2, arXiv:2303.00629 (25 pp).

\bibitem{schur} J. Schur, {\"U}ber die Darstellung der symmetrischen und alternierenden Gruppe durch gebrochene lineare Substitutionen, {\em J. Reine Angew. Math.} {\bf 139} (1911), 155--250.

\bibitem{s} J. Stembridge, Shifted tableaux and the projective representations of symmetric groups, {\em Adv. Math.} {\bf 74} (1989), 87--134.

\bibitem{Wales}  D.B. Wales, Some projective representations of $S_{n}$, {\em J. Algebra} 61 (1979), 37--57.

\end{thebibliography}
\end{document}